\documentclass[12pt,a4paper]{amsart}
\usepackage{mathrsfs}
\usepackage{amssymb,amsmath,amsthm,color}
\usepackage{graphicx,mcite}
\usepackage{hyperref}
\usepackage{url}
\usepackage{setspace}
\usepackage{enumerate}
\newtheorem{theorem}{Theorem}[section]

\newtheorem{proof of lemma}[theorem]{Proof of Lemma}
\newtheorem{proposition}[theorem]{Proposition}

\newtheorem{corollary}[theorem]{Corollary}

\theoremstyle{definition}
\newtheorem{definition}[theorem]{Definition}

\newtheorem{remark}[theorem]{Remark}

\numberwithin{equation}{section}

\begin{document}

\title[Uniqueness pairs for the hyperbola]
{Heisenberg uniqueness pairs for the hyperbola}

\author{Deb Kumar Giri and Rama Rawat}

\address{Deb Kumar Giri, Department of Mathematics, Indian Institute of Science, Bangalore-560012, India.}
\email{debkumarg@iisc.ac.in}


\address{Rama Rawat, Department of Mathematics, Indian Institute of Technology, Kanpur-208016, India.}
\email{rrawat@iitk.ac.in}

\subjclass[2010]{Primary 42A10, 42B10; Secondary 37A45}

\date{\today}

\keywords{Ergodic theory, Fourier transform, Perron-Frobenius operator, Uncertainty principle.}

\begin{abstract}
Let $\Gamma$ be the hyperbola $\{(x,y)\in\mathbb R^2 : xy=1\}$ and $\Lambda_\beta$ 
be the lattice-cross defined by $\Lambda_\beta=\left(\mathbb Z\times\{0\}\right)\cup\left(\{0\}\times\beta\mathbb Z\right)$ in $\mathbb R^2,$ where $\beta$ is a positive real. A result of Hedenmalm and Montes-Rodr\'iguez 
says that $\left(\Gamma,\Lambda_\beta\right)$ is a Heisenberg uniqueness pair if and 
only if $\beta\leq1.$ In this paper, we show that for a rational perturbation of $\Lambda_\beta,$  
namely 
\[\Lambda_\beta^\theta=\left((\mathbb Z+\{\theta\})\times\{0\}\right)\cup\left(\{0\}\times\beta\mathbb Z\right),\]
where $\theta=1/{p},~\text{for some}~{p}\in\mathbb N$ and $\beta$ is a positive real,   
the pair $\left(\Gamma,\Lambda_\beta^\theta\right)$ is a Heisenberg uniqueness pair if 
and only if $\beta\leq{p}.$
\end{abstract}

\maketitle

\section{Introduction}\label{section1}
The notion of Heisenberg uniqueness pair has been around for some time now. 
It was first introduced in \cite{HR} by Hedenmalm and Montes-Rodr\'iguez as 
a variant of the uncertainty principle for Fourier transform, which says that a 
nonzero function and its Fourier transform cannot be simultaneously too small   
(for more on uncertainty principle see \cite{B},\cite{HJ},\cite{Hei}). In \cite{HR}, 
authors had proposed the following.  
Let $\Gamma$ be a finite disjoint union of smooth curves in $\mathbb R^2$ and $\Lambda$ 
be any subset of $\mathbb R^2.$ Let $X(\Gamma)$ be the space of all finite complex-valued 
Borel measures $\mu$ in $\mathbb R^2$ which are supported on $\Gamma$ and are absolutely continuous 
with respect to the arc length measure on $\Gamma.$ For $(\xi,\eta)\in\mathbb R^2,$ 
the Fourier transform of $\mu$ is defined by
\begin{equation}\label{eq01}
\hat\mu{(\xi,\eta)}=\int_\Gamma e^{\pi i(x\xi+ y\eta)}d\mu(x,y).
\end{equation}

\begin{definition}\label{def1}
A pair $\left(\Gamma, \Lambda\right)$ is said to be a Heisenberg uniqueness 
pair (HUP), if the only $\mu\in X(\Gamma)$ satisfying $\hat\mu{(\xi,\eta)}=0$ 
for all $(\xi,\eta)\in\Lambda,$ is the zero measure.
\end{definition}
\smallskip

Thus if $\left(\Gamma, \Lambda\right)$ is a HUP, then the set $\Lambda$ \textit{determines}  
$\mu\in X(\Gamma).$ In interesting examples of HUP, the set $\Lambda,$ for a given 
$\Gamma,$ is taken to be a \textit{very "small"} subset of Euclidean plane. This constraint 
on $\Lambda$ (for a given $\Gamma$) seems to make Heisenberg uniqueness pairs to be 
on the other extreme of uncertainty principles and makes investigating them a challenging problem.  
\smallskip

The expression (\ref{eq01}) of Fourier transform is not the standard one, but since 
the problem we take up here is inspired by and closely follows the methods of \cite{HR}, 
we will follow definitions and notation from \cite{HR} as far as possible. 
Moreover, the definition of Heisenberg uniqueness pair can be and in fact has been 
extended to more general settings as in the work (\cite{Gon,JK,L,S1,S2,Sri}), 
where several examples of Heisenberg uniqueness pairs have been obtained.

\smallskip

\textit{Invariance properties.} The properties of Fourier transform together 
with change of variables imply the following invariance properties for the Heisenberg 
uniqueness pair: 
\smallskip

\begin{enumerate}[(i)]
\item Let $u, v\in\mathbb R^2.$ The pair $(\Gamma,\Lambda)$
is a HUP if and only if the pair $(\Gamma+ \{u\},\Lambda+\{v\})$ is a HUP.

\smallskip

\item Let $T : \mathbb R^2\rightarrow \mathbb R^2$ be an invertible
linear transform and $T^\ast$ be its adjoint. The pair $(\Gamma,\Lambda)$
is a HUP if and only if the pair $\left(T^{-1}\Gamma,T^\ast\Lambda\right)$  is a HUP.
\end{enumerate}
\smallskip

{\bf\textit{Dual condition for HUP.}} For $\zeta\equiv(\xi,\eta)\in\Lambda,$ define 
a function $e_\zeta$ on $\Gamma$ by $e_\zeta(x,y)=e^{\pi i(x\xi+y\eta)}.$ Then the 
pair $(\Gamma,\Lambda)$ is a HUP if and only if the set $\{e_{\zeta}:~\zeta\in\Lambda\}$ 
spans a weak-star dense subspace of $L^{\infty}(\Gamma).$

\smallskip

Our focus on this paper is on a result from \cite{HR}, where authors have studied 
the pair (hyperbola, some lattice-cross) and have proved the following result. 
 
\begin{theorem}\cite{HR}\label{th2}
Let $\Gamma=\{(x,y)\in\mathbb R^2 : xy=1\}$ be the hyperbola and $\Lambda_\beta$ be the 
lattice-cross 
$\Lambda_{\beta}=\left(\mathbb Z\times\{0\}\right)\cup\left(\{0\}\times\beta\mathbb Z\right),$
where $\beta$ is a positive real. Then $\left(\Gamma,\Lambda_\beta\right)$ is a Heisenberg 
uniqueness pair if and only if $\beta\leq1$.
\end{theorem}

By duality, Theorem \ref{th2} is equivalent to the following density result.  
 
\smallskip

\begin{theorem}\cite{HR}
The space of all linear span of the functions 
$\{e_n(x)=e^{\pi inx};~n\in\mathbb Z\}\cup\{e^{\langle\beta\rangle}_n(x)
=e^{\pi in\beta/x};~n\in\mathbb Z\}$ is weak-star dense in $L^\infty(\mathbb R)$ 
if and only if $0<\beta\leq{1}.$
\end{theorem}

Many examples of Heisenberg uniqueness pair have been obtained in the setting of the plane as 
well as in the higher dimensional Euclidean spaces. For more details see (\cite{Ba,Bag,CGGS,GR,GR2,Gon,GJ,HR2,HR3,JK,L,S1,S2,Sri}).
\smallskip

Next, we state the main result of this paper which is a variant of Theorem \ref{th2} and  
the concerned problem finds a mention in (section 7, open problem $(a),$ \cite{HR}). 

\begin{theorem}\label{th3}
Let $\Gamma=\{(x,y)\in\mathbb R^2 : xy=1\}$ be the hyperbola and $\Lambda_\beta^\theta$ be 
the lattice-cross 
\begin{equation}\label{eq100}
\Lambda_\beta^\theta=\left((\mathbb Z+\{\theta\})\times\{0\}\right)\cup\left(\{0\}\times\beta\mathbb Z\right),
\end{equation}
where $\theta=1/{p},~\text{for some}~{p}\in\mathbb N$ and $\beta$ is a positive real.  
Then $\left(\Gamma,\Lambda_\beta^\theta\right)$ is a Heisenberg uniqueness pair if 
and only if $\beta\leq{p}.$
\end{theorem}

\begin{corollary}\label{cor1}
Let $\Gamma=\{(x,y)\in\mathbb R^2 : xy=1\}$ be the hyperbola and $\Lambda_\beta^\theta$ 
be the set $\Lambda_\beta^\theta=\left((\mathbb Z+\{\theta\})\times\{0\}\right)\cup\left(\{0\}\times\beta\mathbb Z\right),$ where $\theta=q/p,~\text{for some}~p\in\mathbb N~\text{and}~q\in\mathbb Z$ with $\emph{\text{gcd(p,q)=1}}$ 
and $\beta$ is a positive real.  Then $\left(\Gamma,\Lambda_\beta^\theta\right)$ is a 
Heisenberg uniqueness pair if and only if $\beta\leq{p}.$
\end{corollary}
\smallskip

\begin{remark}
\begin{enumerate}[(i)]
\item The definition of HUP, may be extended to include more general measures $\mu$'s, 
than only those which are absolutely continuous with respect to the arc length measure 
on $\Gamma.$ But for our purpose, it is essential that the measure  
$\mu$ is absolutely continuous with respect to the arc length measure, for without it  
Theorem \ref{th3} is not true. To see this, we consider the singular measure 
$\mu=\delta_u-\delta_v$ on $\Gamma=\{(x,y)\in\mathbb R^2 : xy=1\},$ where $\delta_s$ 
denotes the point mass at $s\in\Gamma$ and $u=(u_0,1/{u_0})\in\Gamma,$ $v=(v_0,1/{v_0})\in\Gamma.$  
Then it is clear from (\ref{eq01}) that 
\[\hat\mu{(\xi,0)}=e^{\pi i\xi u_0}-e^{\pi i\xi v_0}~~\text{and}~~\hat\mu{(0,\eta)}
=e^{\pi i\eta/u_0}-e^{\pi i\eta/v_0}.\] A simple computation shows that there are  
many different points like $u_0,v_0\in\mathbb R\setminus\{0\}$ such that $\hat\mu|_{\Lambda_\beta^\theta}=0.$
\smallskip

\item The presence of $\theta$ showing up in the condition of $\beta$ is 
somewhat unexpected. The proof of Theorem \ref{th3} is a careful adaptation 
of methods in \cite{HR}. The question is still open when $\theta$ is irrational. 
\smallskip 

\item For any polynomial $P$ in $\mathbb R^2,$ the Fourier transform of $\mu$ satisfies 
\[P\left(\dfrac{\partial_\xi}{i\pi}, \dfrac{\partial_\eta}{i\pi}\right)\hat\mu{( \xi,\eta)}
=\int_{\mathbb R^2} e^{i\pi(x\xi+ y\eta)}P(x,y)d\mu(x,y)\] 
in the sense of distributions. In particular, if $P(x,y):=xy-1$ and $\Gamma$ is the 
zero set of $P,$ i.e., the hyperbola $xy=1,$ then any $\mu\in X(\Gamma),$ $u:=\hat\mu$ is a solution   
of the one-dimensional Klein-Gordon equation   
\[\left(\partial_\xi\partial_\eta+\pi^2\right)u(\xi,\eta)=0.\]
Therefore, $\left(\Gamma,\Lambda\right)$ is a HUP if and only if the above partial differential equation, 
with the initial condition $u=0$ on $\Lambda$ has only a trivial solution. 

\smallskip

\item Let $\mathcal{AC}(\Gamma,\Lambda)=\{\mu\in X(\Gamma) : \widehat\mu|_{\Lambda}=0\}.$ 
Let $\Gamma, \Lambda_\beta^\theta$ be as in the Theorem \ref{th3}. Then 
$\mathcal{AC}(\Gamma,\Lambda_\beta^\theta)=\{0\}$ if and only if $\beta\leq{p}.$ 
For $p<\beta<\infty,$ it seems likely, that $\mathcal{AC}(\Gamma,\Lambda_\beta^\theta)$ 
is infinite-dimensional in analogy with the results in \cite{MHR}. This question 
is still open. 
\end{enumerate}
\end{remark}

\section{Proof of the main result}

Let $\Gamma=\{(x,y)\in\mathbb R^2 : ~xy=1\}$ be the hyperbola and $\mu\in X(\Gamma),$ 
then there exists $f\in L^1(\mathbb R, \sqrt{1+1/t^4}~dt)$ 
such that for bounded and continuous function $\varphi$ on $\mathbb R^2,$
\[\int_\Gamma\varphi(x,y)d\mu(x,y)=\int_{\mathbb R\setminus\{0\}}\varphi(t,1/t)f(t)\sqrt{1+1/t^4}dt.\] 
In particular, the Fourier transform of $\mu$ can be expressed as 
\[\hat\mu{(\xi_1,\xi_2)}=\int_{\mathbb R\setminus\{0\}} e^{\pi i(\xi_1 t+ \xi_2/t)}g(t)dt~~ 
\text{for}~~(\xi_1,\xi_2)\in\mathbb R^2,\] 
where $g(t)=f(t)\sqrt{1+1/t^4}.$ Let $\mathcal F$ be the subspace of all linear span of the functions 
$\{e^{p}_n(x)=e^{\pi i(n+1/p)x};~n\in\mathbb Z\}\cup\{e^{\beta}_n(x)
=e^{\pi in\beta/x};~n\in\mathbb Z\}$ in $L^\infty(\mathbb R),$  
where $p\in\mathbb N$ and $\beta$ is a positive real. By duality, 
Theorem \ref{th3} is equivalent to the following density result.
\smallskip

\begin{theorem}\label{th4}
The set $\mathcal F$ is weak-star dense in $L^\infty(\mathbb R)$ 
if and only if $0<\beta\leq{p}.$
\end{theorem}

\subsection{Proof of the necessary condition of Theorem \ref{th4}}

\begin{proof}
The proof relies on the fact that bounded harmonic functions on the upper 
half-plane $\mathbb C_+=\{z\in\mathbb C : ~\text{Im}~z>0\}$ separate points 
of $\mathbb C_+.$ 

The functions $e^{p}_n;~n\in\mathbb Z$ 
and $e^{\beta}_n;~n\in\mathbb Z$ extend to bounded harmonic functions on $\mathbb C_+,$ 
by letting

\[e^{p}_n(z)=\begin{cases} 
      e^{\pi i(n+1/p)z},~\text{Im~}z>0,~n\geq0 \\
      e^{\pi i(n+1/p)\bar z},~\text{Im}~z>0,~n<0, 
   \end{cases}
\]

and 
\[ e^{\beta}_n(z)=\begin{cases} 
      e^{\pi in\beta/\bar z},~\text{Im}~z>0,~n\geq0 \\
      e^{\pi in\beta/z},~\text{Im}~z>0,~n<0. 
   \end{cases}
\]

\smallskip

For $f\in L^\infty(\mathbb R)$ the bounded harmonic function extension of $f$ on 
$\mathbb C_+$ is defined via point evaluation with respect to the Poisson kernel $P_z:$
\[f(z)=\frac{1}{\pi}\int_{\mathbb R}f(t)P_z(t)dt,\] where 
$z=x+iy\in\mathbb C_+,$ and $P_z(t)=\frac{y}{(x-t)^2+y^2},~t\in\mathbb R.$ As 
$P_z(\cdot)\in L^1(\mathbb R),$ for $z\in\mathbb C_+,$ the point evaluation 
$f\rightarrow f(z)$ defines a weak-star continuous functional on $L^\infty(\mathbb R).$  
\smallskip

Moreover it is known that all bounded harmonic functions on  $\mathbb C_+$  arise as 
bounded harmonic extensions of functions in $L^\infty(\mathbb R).$ Therefore, if $\mathcal F$ 
is weak-star dense in $L^\infty(\mathbb R),$ then $\mathcal F$ must separate points of  $\mathbb C_+.$ 
\smallskip

Solving for $z_1,z_2\in\mathbb C_+$ such that 
\begin{equation}\label{eq20}
\begin{array}{lr}
e^{p}_n(z_1)=e^{p}_n(z_2)\\
e^{\beta}_n(z_1)=e^{\beta}_n(z_2)
\end{array} \text{for~all} ~n\in\mathbb Z,
\end{equation}
\smallskip

 
 we get
  
\[ \begin{cases} 
      z_1-z_2\in2{p}\mathbb Z, \\
      \frac{1}{z_1}-\frac{1}{z_2}\in\frac{2}{\beta}\mathbb Z. 
   \end{cases}
\]
In view of (Lemma 1.5, \cite{HR}), with $m=n=1,~a=2{p}$ and $b=2/\beta,$ we therefore have 
\[z_1={p}\left(1+\text{i}\sqrt{\beta/{p}-1}\right),~~z_2={p}\left(-1+\text{i}\sqrt{\beta/{p}-1}\right),\] 
which are two distinct points of $\mathbb C_+$ if and only if ${p}<\beta.$  
As a consequence if ${p}<\beta,$  $\mathcal F$ 
does not separate points of $\mathbb C_+.$ 

Thus the condition $0<\beta\leq{p}$ is necessary for 
$\left(\Gamma,\Lambda_\beta^\theta\right)$ to be a Heisenberg uniqueness pair.
\end{proof}

\subsection{Proof of sufficient condition in Theorem \ref{th4}}

To prove the sufficient condition in Theorem \ref{th4}, we need the following interlude  
on Gauss-type maps and invariant measures. In what follows we assume that $0<\beta\leq p.$  
\smallskip

\textit{A Gauss-type map.} For $u\in\mathbb R,$ the expression $\{u\}_2$ is the 
unique number in $(-1,1]$ such that $u-\{u\}_2\in2\mathbb Z.$ We consider a 
Gauss-type map $\tau$ on the interval $(-{p},{p}],$ which is topologically equivalent 
to $\mathbb R/{2{p}\mathbb Z}.$ The map $\tau$ is defined by letting 

\[\tau(x)=\begin{cases} 
      \left\{-\frac{{p}}{x}\right\}_2,~x\neq0 \\
      \quad 0 \qquad x=0. 
   \end{cases}
\] 

Observe that, for $j\in\mathbb Z^\ast=\mathbb Z\setminus\{0\},$ $\tau$ can be 
expressed as \[\tau(x)=-\frac{{p}}{x}+2j~\text{whenever}~\frac{{p}}{2j+1}<x\leq\frac{{p}}{2j-1},\] 
and hence $\tau: \left(\frac{{p}}{2j+1},\frac{{p}}{2j-1}\right]\rightarrow(-1,1]$ 
is one-to-one and for $x\in(-p,p]\setminus\frac{p}{2\mathbb Z+1},$ the derivative of $\tau$ 
is $\tau'(x)=\frac{p}{x^2}.$ Next, we define the map 
$U : (-p,p]\rightarrow(-p,p]$ by letting 
\begin{equation}
U(x)=p\tau(x).
\end{equation}

\textit{Invariant measures for Gauss-type map.} If $\varphi$ is a continuous $2p$-periodic 
function on $\mathbb R$ and $\nu$ is a finite complex-valued Borel measure on $(-p,p],$ 
then the integral 
\begin{equation}\label{eq39}
\int_{(-p,p]}\varphi(x)d\nu(x)
\end{equation}
is well-defined. A function $\varphi$ on $(-p,p]$ is said to be \textit{pseudo-continuous},  
if for every open subset $\mathcal O$ of $(-p,p]$ such that $(-p,p]\setminus\mathcal O$ is countable, 
$\varphi$ is bounded on $(-p,p]$ and continuous on $\mathcal O.$ In fact, (\ref{eq39}) 
makes sense if $\varphi$ is \textit{pseudo-continuous}. Also observe that for a 
\textit{pseudo-continuous} function $\varphi$ on $(-p,p],$ $\varphi\circ U$ is 
\textit{pseudo-continuous}. 
\smallskip

A finite complex Borel measure $\nu$ on $(-p,p]$ is $U$-invariant provided that 
\begin{equation}\label{eq49}
\int_{(-p,p]}\varphi\left(U(x)\right)d\nu(x)=\int_{(-p,p]}\varphi(x)d\nu(x)
\end{equation}
holds for all \textit{pseudo-continuous} functions $\varphi.$ For details about invariant 
measures for continuous maps see (P. 97, \cite{EW}). Rewriting the left hand side of (\ref{eq49}), we get 
\begin{eqnarray}
\nonumber\int_{(-{p},{p}]\setminus\{0\}}\varphi\left(U(x)\right)d\nu(x) 
&=& \sum\limits_{j\in\mathbb Z^\ast}\int_{\left(\frac{{p}}{2j+1},\frac{{p}}{2j-1}\right]}
\varphi\left({p} \tau(x)\right)d\nu(x)\\\nonumber
&=& \sum\limits_{j\in\mathbb Z^\ast}\int_{\left(\frac{{p}}{2j+1},\frac{{p}}{2j-1}\right]}
\varphi\left({p}\left(-\frac{{p}}{x}+2j\right)\right)d\nu(x)\\\nonumber
&=& \sum\limits_{j\in\mathbb Z^\ast}\int_{(-{p},{p}]}\varphi\left(t\right)
d\nu\left(\frac{{p}^2}{2{p} j-t}\right)\\\nonumber
&=& \sum\limits_{j\in\mathbb Z^\ast}\int_{(-{p},{p}]}\varphi\left(t\right)d\nu_j(t),
\end{eqnarray}
where $d\nu_j(t)=d\nu\left(\frac{{p}^2}{2{p} j-t}\right).$
Thus it follows that $\nu$ is $U$-invariant provided that  
\[\nu=\nu({0})\delta_0+\sum\limits_{j\in\mathbb Z^\ast}\nu_j.\] 
In addition, given $\lambda\in\mathbb C,$ complex Borel measure $\nu$ 
on $(-p,p]$ is $(U,\lambda)$-invariant provided that
\[\int_{(-p,p]}\varphi\left(U(x)\right)d\nu(x)=\lambda\int_{(-p,p]}\varphi(x)d\nu(x)\]
holds for all \textit{pseudo-continuous} functions $\varphi.$ Equivalently $\nu$ is 
$(U,\lambda)$-invariant if  
\[\lambda\nu=\nu({0})\delta_0+\sum\limits_{j\in\mathbb Z^\ast}\nu_j.\]
It is easy to see that, for $|\lambda|>1,$ there are no $(U,\lambda)$-invariant measures 
except for the zero measure.
\smallskip

\textit{Invariant measure of infinite mass.} Consider the positive $\sigma$-finite 
measure $\omega$ on $(-p,p]$ of infinite mass by letting 
\[d\omega(x)=\frac{dx}{{p}^2-x^2}.\] Then the measure $\omega$ is $U$-invariant as  
\[d\omega_j(t)=d\omega\left(\frac{{p}^2}{2{p} j-t}\right)=\frac{dt}{(2{p} j-t)^2-{p}^2}~\text{and}~ 
\sum\limits_{j\in\mathbb Z^\ast}\frac{1}{(2{p} j-t)^2-{p}^2}=\frac{1}{{p}^2-t^2},\]
and therefore $\omega$ satisfies \[\omega=\omega({0})\delta_0+\sum\limits_{j\in\mathbb Z^\ast}\omega_j.\]
\smallskip

Let $|U| : [0,p]\rightarrow [0,p]$ be the map $|U|(x)=|U(x)|.$ It can be shown that 
the measure $\omega$ is $|U|$-invariant. Also $|U|$ is ergodic. As a consequence of 
the \textit{Birkhoff ergodic theorem} we get the following result.   
\smallskip

\begin{proposition}\label{prop1}
Suppose $\lambda\in\mathbb C$ with $|\lambda|=1.$ Let $\nu$ be a absolutely 
continuous finite complex $(U,\lambda)$-invariant Borel measure on $(-{p},{p}].$ 
Then $\nu$ is zero.
\end{proposition} 
The proof of Proposition \ref{prop1} works along the same lines as in proof of 
(Proposition 2.4, \cite{HR}).
\smallskip

\textit{Gauss-type map for $0<\beta<p$.}  
The Gauss-type map $\tau_\beta$ is defined by $\tau_\beta(0)=0,$ 
and for $x\neq0,$ \[\tau_\beta(x)=\left\{-\frac{\beta}{x}\right\}_2,\]
and the map $U_\beta : (-p,p]\rightarrow(-p,p]$ which is associated to $\tau_\beta$ 
is defined by letting 
\begin{equation}
U_\beta(x)=p\tau_\beta(x).
\end{equation} 
\smallskip

\begin{proposition}
For $\lambda\in\mathbb C$ with $|\lambda|=1,$ a finite complex Borel measure $\nu$ 
on $(-p,p]$ is $(U_\beta,\lambda)$-invariant provided that
\[\lambda\nu=\nu({0})\delta_0+\sum\limits_{j\in\mathbb Z^\ast}\nu_j,\] 
where $d\nu_j(t)=d\nu\left(\frac{p\beta}{2{p} j-t}\right),$ $t\in (-p,p].$
\end{proposition}

\begin{proof}
A finite complex Borel measure $\nu$ on $(-{p},{p}]$ is $(U_\beta,\lambda)$-invariant provided that 
\[\int_{(-{p},{p}]}\varphi\left(U_\beta(x)\right)d\nu(x)=\lambda\int_{(-{p},{p}]}\varphi(x)d\nu(x)\] 
holds for all pseudo-continuous test functions $\varphi$ on $(-p,p].$ Observe that 
\begin{eqnarray}
\nonumber\int_{(-{p},{p}]\setminus\{0\}}\varphi\left(U_\beta(x)\right)d\nu(x) 
&=& \sum\limits_{j\in\mathbb Z^\ast}\int_{\left(\frac{\beta}{2j+1},\frac{\beta}{2j-1}\right]}
\varphi\left({p} \tau_\beta(x)\right)d\nu(x)\\\nonumber
&=& \sum\limits_{j\in\mathbb Z^\ast}\int_{\left(\frac{\beta}{2j+1},\frac{\beta}{2j-1}\right]}
\varphi\left({p}\left(-\frac{\beta}{x}+2j\right)\right)d\nu(x)\\\nonumber
&=& \sum\limits_{j\in\mathbb Z^\ast}\int_{(-{p},{p}]}\varphi\left(t\right)
d\nu\left(\frac{p\beta}{2{p} j-t}\right)\\\nonumber
&=& \sum\limits_{j\in\mathbb Z^\ast}\int_{(-{p},{p}]}\varphi\left(t\right)d\nu_j(t).
\end{eqnarray}
Thus a measure $\nu$ is $(U_\beta,\lambda)$-invariant only if 
\[\lambda\nu=\nu({0})\delta_0+\sum\limits_{j\in\mathbb Z^\ast}\nu_j.\] 
\end{proof}

\smallskip

{\bf\textit{The sum space.}}  Let $L^\infty_{{p}}(\mathbb R)$ denote the subspace of 
all functions $f\in L^\infty(\mathbb R)$ such that the map $x\longmapsto e^{-\pi ix/p}f(x)$  
is 2-periodic. Then the weak-star closure in $L^\infty(\mathbb R)$ of linear 
span of $\{e^{p}_n(x)=e^{\pi i(n+1/p)x};~n\in\mathbb Z\}$ equals $L^\infty_{{p}}(\mathbb R).$
\smallskip

Similarly, let $L^\infty_{\beta}(\mathbb R)$ denote the subspace of all functions 
$f\in L^\infty(\mathbb R)$ such that the map $x\longmapsto f(\beta/x)$ is 2-periodic. 
Then the weak-star closure in $L^\infty(\mathbb R)$ of linear span of  
$\{e^{\beta}_n(x)=e^{\pi in\beta/x};~n\in\mathbb Z\}$ equals $L^\infty_{\beta}(\mathbb R).$ 
\smallskip

Consider the sum space $L^\infty_{{p}}(\mathbb R)+L^\infty_{\beta}(\mathbb R).$ 
The sufficient part of Theorem \ref{th4} is equivalent to the following density result.
\smallskip

\begin{theorem}\label{th5}
For $0<\beta\leq{p},$ the sum space \[L^\infty_{{p}}(\mathbb R)+L^\infty_{\beta}(\mathbb R)\] 
is weak-star dense in $L^\infty(\mathbb R).$ 
\end{theorem}
\smallskip

\textit{The composition operator.} Observe that the functions in $L^\infty_{{p}}(\mathbb R)$ 
are defined freely on $(-{p},{p}],$ and because of periodicity they are uniquely determined 
on $\mathbb R\setminus(-{p},{p}].$ Similarly, the functions in $L^\infty_{\beta}(\mathbb R)$ 
are defined freely on $\mathbb R\setminus(-\beta,\beta],$ and because of periodicity they are  
uniquely determined on $(-\beta,\beta].$ For $E\subseteq\mathbb R,$ the function $\chi_E$ 
denotes the characteristic function of $E$ on $\mathbb R.$
\smallskip

Define operator $\text{S} : L^\infty((-{p},{p}])\rightarrow L^\infty(\mathbb R\setminus(-{p},{p}])$  
by   
\[\text{S}[\varphi](x)=\varphi\left(p\left\{x/p\right\}_2\right)\chi_{\mathbb R\setminus(-{p},{p}]}(x),\]
where $\varphi\in L^\infty((-{p},{p}])$ and $x\in\mathbb R.$
\smallskip

The operator $\text{T}_{\beta} : L^\infty(\mathbb R\setminus(-\beta,\beta])\rightarrow L^\infty((-\beta,\beta])$ 
is defined by  
\[\text{T}_{\beta}[\psi](x)=\psi\left(\frac{\beta}{\{\beta/x\}_2}\right)\chi_{(-\beta,\beta]}(x),\] 
where $\psi\in L^\infty(\mathbb R\setminus(-\beta,\beta])$ and $x\in\mathbb R.$
A simple computation shows that for $\psi\in L^\infty(\mathbb R\setminus(-\beta,\beta]),$
\[\text{S}[\psi\circ I_\beta]\circ I_\beta=\text{T}_{\beta}[\psi],\] 
where $I_\beta(x)=-\beta/x.$ Then the composition operator $\text{T}_{\beta}\text{S} : L^\infty((-{p},{p}])\rightarrow L^\infty((-{p},{p}])$ is defined as   
\[\text{T}_{\beta}\text{S}[\varphi](x)=\varphi\left({p}\left\{\frac{\beta_0}{\{\beta/x\}_2}\right\}_2\right)\chi_{E_{\beta}}(x),\] 
for $\varphi\in L^\infty((-{p},{p}]),$ where $\beta_0=\beta/{p}$ and 
$E_{\beta}=\left\{x\in(-\beta,\beta]\setminus\{0\}: \frac{\beta_0}{\{\beta/x\}_2}\in\mathbb R\setminus(-1,1]\right\}.$
\smallskip

\textit{The Koopman operator.}
For $0<\beta\leq p,$ consider the weighted Koopman operator $\mathcal C_\beta : L^\infty((-{p},{p}])\rightarrow L^\infty((-{p},{p}])$ associated to $U_\beta$ by \[\mathcal C_\beta[\varphi](x)=\varphi\circ U_\beta(x)\chi_{(-\beta,\beta]}(x),\] 
where $x\in\mathbb R.$ The predual adjoint $\mathcal C_\beta^\ast$ of $\mathcal C_\beta$ is the Perron-Frobenius 
operator $\mathcal P_\beta : L^1((-p,p])\rightarrow L^1((-{p},{p}])$ given by 
\[\mathcal P_\beta [f](x)=\sum\limits_{j\in\mathbb Z^\ast}\frac{p\beta}{(2pj-x)^2}f\left(\frac{p\beta}{2pj-x}\right).\]
The operator $\mathcal P_\beta$ is linear and a norm contraction, that is, 
\[\|\mathcal P_\beta [f]\|_{L^1((-p,p])}\leq\|f\|_{L^1((-p,p])},\] 
for $f\in L^1((-p,p]).$ Thus the point spectrum $\sigma_{point}(\mathcal P_\beta)$ 
of $\mathcal P_\beta$ is contained in the closed unit disk $\bar{\mathbb D}=\{\lambda\in\mathbb C : |\lambda|\leq1\}.$

A simple computation gives $\text{T}_{\beta}\text{S}=\mathcal C_\beta^2.$
Thus the operator $\text{I}-\text{T}_{\beta}\text{S} : L^\infty((-{p},{p}])\rightarrow L^\infty((-{p},{p}])$ 
can be expressed as 
\begin{equation}\label{eq19}
\text{I}-\text{T}_{\beta}\text{S}=(\text{I}+\mathcal C_\beta)(\text{I}-\mathcal C_\beta).
\end{equation}

\begin{proposition}\label{prop3}
The range of the operator $\text{I}-\text{T}_{\beta}\text{S}$ is weak-star dense in 
$L^\infty((-{p},{p}]).$ 
\end{proposition}

\begin{proof}
In view of the identity (\ref{eq19}), we prove that the operators $\text{I}+\mathcal C_\beta$ and 
$\text{I}-\mathcal C_\beta$ have weak-star dense range. For $\lambda=\pm1,$ 
the weak-star closure of the range 
of $\lambda \text{I}-\mathcal C_\beta$ equals to $L^\infty((-{p},{p}])$ if and only if the predual adjoint 
\[\lambda \text{I}-\mathcal P_\beta : L^1((-{p},{p}])\rightarrow L^1((-{p},{p}])\] 
has null kernel. Therefore, it remains to show that $\pm1$ are not eigenvalues of $\mathcal P_\beta.$ 
\smallskip

We actually prove a stronger statement, namely that if $\lambda\in\mathbb C$ 
with $|\lambda|=1,$ then $\lambda$  is not an eigenvalue of $\mathcal P_\beta.$

\textit{Case 1.} Assume that $\beta=p$ and recall that $\mathcal C_p[\varphi](x)=\varphi\circ U(x)\chi_{(-p,p]}(x).$ 
By the dual relation, we get 
\[\int_{(-p,p]}\varphi(x)\mathcal P_p[\psi](x)dx=\int_{(-p,p]}\varphi\circ U(x)\psi(x)dx.\]
Suppose $\psi_0$ is a non-zero eigenfunction of $\mathcal P_p$ with eigenvalue 
$\lambda$ such that $|\lambda|=1,$ then $\mathcal P_p[\psi_0]=\lambda\psi_0$ and therefore 
\[\lambda\int_{(-p,p]}\varphi(x)\psi_0(x)dx=\int_{(-p,p]}\varphi(U(x))\psi_0(x)dx.\]
Hence a non-zero absolutely continuous finite Borel measure $d\nu(x)=\psi_0(x)dx$  
is $(U,\lambda)$-invariant which contradicts Proposition \ref{prop1}. Thus 
$|\lambda|=1$ is not an eigenvalue of $\mathcal P_p.$

\textit{Case 2.} Assume that $0<\beta<p.$ By dual relation we get 
\begin{equation}\label{eq11}
\lambda\int_{(-p,p]}\varphi(x)d\nu(x)=\int_{(-\beta,\beta]}\varphi(p\tau_\beta(x))d\nu(x)
\end{equation}
for all $\varphi\in L^\infty((-p,p]),$ where $\lambda\in\mathbb C$ with $|\lambda|=1$ and 
$\nu$ is a absolutely continuous finite Borel measure on $(-p,p].$  
It follows that \[\lambda d\nu(x)=\sum\limits_{j\in\mathbb Z^\ast}d\nu_j(x),~~x\in (-p,p],\] 
where $d\nu_j(x)=d\nu\left(\frac{\beta p}{2pj-x}\right).$ Taking 
absolute values and then integrating over $(-p,p],$ we have 
\[\int_{(-p,p]}|d\nu(x)|\leq\sum\limits_{j\in\mathbb Z^\ast}\int_{(-p,p]}|d\nu_j(x)|=
\int_{(-\beta,\beta]}|d\nu(x)|,~~x\in (-p,p],\] which is only possible if the following holds :  
\[|d\nu(x)|=\sum\limits_{j\in\mathbb Z^\ast}|d\nu_j(x)|,~~x\in (-p,p].\] From the above 
equality, it follows that 
\[d\nu(x)=0~\text{for}~x\in (-p,p]\setminus(-\beta,\beta].\]
Iterating the Equation \ref{eq11} and using the above arguments repeatedly, we conclude that 
\[d\nu(x)=0~\text{for}~x\in (-p,p]\setminus E_\beta(n),\] where 
$E_\beta(n)=\{x\in(-p,p] : ~U^{(k)}_\beta(x)\in(-\beta,\beta]~\text{for}~k=0,\cdots,n-1\}$ 
and $U^{(k)}_\beta$ the $k$-th iteration of $U_\beta.$
By letting $n\rightarrow +\infty$ we get that 
\[d\nu(x)=0~\text{for}~x\in (-p,p]\setminus E_\beta(\infty)\] with 
\[E_\beta(\infty)=\{x\in(-p,p] : ~U^{(k)}_\beta(x)\in(-\beta,\beta]~\text{for}~k=0,1,2,\cdots\}.\]
The set $E_\beta(\infty)$ is $U_\beta$ invariant and it has zero length.  
Hence the absolutely continuous measure $\nu$ vanishes almost everywhere on $(-p,p].$  
Thus we conclude that $|\lambda|=1$ is not an eigenvalue of $\mathcal P_\beta.$
\end{proof}

Finally, we show that for $0<\beta\leq{p},$ the sum space is weak-star 
dense in $L^\infty(\mathbb R).$

\begin{proof}[Proof of Theorem \ref{th5}] The proof is similar to the proof 
of (Lemma 5.2, \cite{HR}), hence omitted. This completes the proof of Theorem \ref{th5}.
\end{proof}

\noindent{\bf Acknowledgements.} The first named author gratefully acknowledges the support 
provided by IIT Kanpur, Government of India.

\bigskip


\end{document}